\documentclass[reqno,11pt]{amsart} 

\usepackage{amsfonts,amsmath,latexsym}
\usepackage{amstext,amssymb,esint}
\usepackage{amsthm}
\usepackage {epsf}
\usepackage {epsfig}
\usepackage{tikz}
\usetikzlibrary{arrows, patterns}
\usepackage[scanall]{psfrag}
\usepackage {graphicx}
\usepackage[section]{placeins}
\usepackage{verbatim}
\usepackage{appendix}

\newcommand {\R}{\mathbf{R}}
\newcommand {\Ss}{\mathbf{S}}

\newcommand {\B}{\mathbf{B}}

\newcommand {\dist}{\operatorname {dist}}

\newcommand{\Ric}{\operatorname{Ric}}

\newcommand{\quotes}[1]{``#1''}

\newtheorem {thm} {Theorem}
\newtheorem {prop} [thm] {Proposition}
\newtheorem {lemma} [thm] {Lemma}
\newtheorem {cor} [thm] {Corollary}

\newtheorem {rmk} {Remark}

\title[Compactness of constant $Q$-curvature metrics]{Compactness within the space of complete, constant $Q$-curvature 
metrics on the sphere with isolated singularities}

\author[J.H. Andrade]{Jo\~{a}o Henrique\ Andrade}
\author[J.M. do \'O]{Jo\~ao Marcos do \'O}
\author[J. Ratzkin]{Jesse Ratzkin}

\address[J.H. Andrade]{
	Institute of Mathematics and Statistics,
	University of S\~ao Paulo
	\newline\indent 
	05508-090, S\~ao Paulo-SP, Brazil
	\newline\indent 
	and
	\newline\indent 
	Department of Mathematics,
	Federal University of Para\'{\i}ba 
	\newline\indent 
	58051-900, Jo\~ao Pessoa-PB, Brazil}
\email{{andradejh@ime.usp.br}}
\email{{andradejh@mat.ufpb.br}}

\address[J.M. do \'O]{Department of Mathematics,
	Federal University of Para\'{\i}ba
	\newline\indent 
	58051-900, Jo\~ao Pessoa-PB, Brazil}
\email{{jmbo@pq.cnpq.br}}

\address[J. Ratzkin]{Department of Mathematics,
	Universit\"{a}t W\"{u}rzburg
	\newline\indent
	97070, W\"{u}rzburg-BA, Germany}
\email{{jesse.ratzkin@mathematik.uni-wuerzburg.de}}

\thanks{Research supported in part by Conselho Nacional de Desenvolvimento Cient\'ifico e Tecnol\'ogico (CNPq): grant 305726/2017-0, Coordena\c c\~ao de Aperfei\c coamento de Pessoal de N\'ivel Superior (CAPES): grant 88882.440505/2019-01, and Funda\c c\~ao de Apoio \`a Pesquisa do Estado de S\~ao Paulo (FAPESP): grant 2020/07566-3}
\begin{document}

\begin{abstract} 
In this paper we consider the moduli space of complete, conformally flat
metrics on a sphere with $k$ punctures having constant positive $Q$-curvature 
and positive scalar curvature. Previous work has shown that such metrics admit 
an asymptotic expansion near each puncture, allowing one to define an asymptotic 
necksize of each singular point. We prove that any set in the moduli space 
such that the distances between distinct punctures and the asymptotic necksizes 
all remain bounded away from zero is sequentially compact, mirroring a theorem 
of D. Pollack about singular Yamabe metrics. Along the way we define a radial 
Pohozaev invariant at each puncture and refine some {\it a priori} bounds of the 
conformal factor, which may be of independent interest. 
\end{abstract} 

\maketitle

\section{Introduction} 

In 1960 H. Yamabe \cite{Y} proposed a program to find optimal metrics in a 
conformal class on a manifold of dimension at least three by minimizing  
the total scalar curvature functional, obtaining a constant scalar curvature 
representative in each conformal class. This program, which now bears his name, 
led to many advancements by N. Trudinger \cite{Tru}, T. Aubin \cite{Aub}, 
R. Schoen \cite{Sch}, and 
many others in the understanding of how geometry, topology and analysis 
interact with each other in compact Riemannian manifolds. The reader can find an 
excellent survey of the resolution of the original Yamabe problem in \cite{LP}. The lack of 
compactness of the group of conformal transformations of the sphere presents 
one of the many complications in carrying out Yamabe's program. This same 
lack of compactness forces one to examine singular solutions which blow up 
along a closed subset. Many people continue to study both the 
regular and the singular Yamabe problems, and many open questions 
in both programs remain. More recent results include \cite{KMS}, in 
which the authors prove the set of solutions to the Yamabe 
problem within a conformal class is compact provided $3 \leq n \leq 24$
 and the conformal class is not the one of the round sphere.  

In recent years many people have pursued parts of Yamabe's 
program for other notions of curvature. In the present note, we 
explore a part of the singular Yamabe program as applied to 
the fourth order $Q$-curvature, which is a higher order analog 
of scalar curvature. On a Riemannian manifold $(M,g)$ of dimension 
$n \geq 5$, the $Q$-curvature is 
\begin{equation} \label{q_curv_defn}
Q_g = -\frac{1}{2(n-1)} \Delta_g R_g - \frac{2}{(n-2)^2} |\Ric_g|^2 + 
\frac{n^3-4n^2+16n-16}{8(n-1)^2(n-2)^2} R_g^2, 
\end{equation} 
where $R_g$ is the scalar curvature of $g$, $\Ric_g$ is the 
Ricci curvature of $g$, and $\Delta_g$ is the Laplace--Beltrami 
operator of $g$. After a conformal change, the $Q$-curvature 
transforms as 
\begin{equation} \label{trans_rule1} 
\widetilde g = u^{\frac{4}{n-4}} g \Rightarrow 
Q_{\widetilde g} = \frac{2}{n-4} u^{- \frac{n+4}{n-4}} P_gu, 
\end{equation} 
where $P_g$ is the Paneitz operator 
\begin{eqnarray} \label{paneitz_op_defn} 
P_gu=\Delta_g^2u + \operatorname{div} \left ( 
\frac{4}{n-2} \operatorname{Ric}_g (\nabla u, \cdot) - 
\frac{(n-2)^2 + 4}{2(n-1)(n-2)} R_g \langle \nabla u, \cdot
\rangle \right )+\frac{n-4}{2} Q_g u. 
\end{eqnarray} 
Paneitz \cite{Pan1} first discovered the operator $P_g$ and 
investigated its conformal invariance. Thereafter Branson 
\cite{Bran1, Bran2} began 
a thorough investigation of $Q_g$ and its variants.
The reader can find excellent summaries of the fourth order 
$Q$-curvature in \cite{BG, CEOY, HY}. 

The $Q$-curvature of the round metric $\overset{\circ}{g}$ is 
$\frac{n(n^2-4)}{8}$, and setting $Q_g$ to be this value gives 
the equation 
\begin{equation} 
P_g u = \frac{n(n-4)(n^2-4)}{16} u^{\frac{n+4}{n-4}}.
\end{equation} 
Just as in the scalar curvature setting, one can search for 
constant $Q$-curvature metrics in a conformal class by 
minimizing the total $Q$-curvature. However, because of 
the conformal invariance one encounters the same lack 
of compactness and presence of singular solutions. Hang 
and Yang \cite{HY2} carry out part of this program in the regular case, 
assuming that the background metric also has positive Yamabe 
invariant. 

In any event, a complete understanding of the 
fourth order analog of the Yamabe problem would 
require an understanding of the following singular 
problem: let $(M,g)$ be a compact Riemannian manifold and 
let $\Lambda \subset M$ be a closed subset. A 
conformal metric $\widetilde{g} = u^{\frac{4}{n-4}} g$ 
is a singular constant $Q$-curvature metric if $Q_{\widetilde{g}}$ 
is constant and $\widetilde{g}$ is complete on $M 
\backslash \Lambda$. According to \eqref{trans_rule1} 
we can write this geometric problem as 
\begin{eqnarray} \label{gen_sing_q_curv} 
P_gu & = & \frac{n(n-4)(n^2-4)}{16} u^{\frac{n+4}{n-4}} 
\textrm{ on }M \backslash \Lambda, \\ \nonumber 
\liminf_{x \rightarrow x_0} u(x) & = & \infty \textrm{ for each }
x_0 \in \Lambda.
\end{eqnarray}

For the remainder of our work we concentrate on the 
case that $(M,g) = (\Ss^n, \overset{\circ}{g})$ is the 
round metric on the sphere and $\Lambda = \{ p_1, 
\dots, p_k\}$ is a finite set of distinct points. Thus we 
examine, given a singular set $\Lambda$ with $\#(\Lambda) 
= k$, the set of functions 
$$u : \Ss^n \backslash \Lambda = \Ss^n \backslash 
\{ p_1, \dots, p_k\} \rightarrow (0,\infty)$$ 
that satisfy 
\begin{eqnarray} \label{const_q_curv_eq}
\overset{\circ}{P} u & = & P_{\overset{\circ}{g}} u = 
\frac{n(n-4)(n^2-4)}{16} u^{\frac{n+4}{n-4}} \\ \nonumber 
\liminf_{x \rightarrow p_j} u(x) & = & \infty \textrm{ for each } 
j=1,2, \dots, k.
\end{eqnarray}
For technical reasons we will also require $R_g \geq 0$.

Following 
\cite{MPU} we define the marked moduli space 
\begin{equation} \label{marked_moduli_defn} 
\mathcal{M}_\Lambda = \left \{ g \in [\overset{\circ}{g}] : Q_g = 
\frac{n(n^2-4)}{8}, \ R_g \geq 0, \ g \textrm{ is complete on } \ \Ss^n \backslash \Lambda
\right \} 
\end{equation} 
and the unmarked moduli space 
\begin{eqnarray} \label{unmarked_moduli_defn} 
\mathcal{M}_k=\left\{ g \in [\overset{\circ}{g}]: Q_g = \frac{n(n^2-4)}{8}, \ R_g \geq 0, \
g \textrm{ is complete on } \ \Ss^n \backslash \Lambda, \ \# (\Lambda) = k \right\}.
\end{eqnarray} 
We equip each moduli space with 
the Gromov--Hausdorff topology. C. S. Lin \cite{Lin} proved that $\mathcal{M}_1$ is 
the empty set, and recently Frank and K\"onig \cite{FK} classified all metrics in 
$\mathcal{M}_2$, proving 
$$\mathcal{M}_{p,q} = (0,\epsilon_n] \textrm{ for each pair } p \neq q \in \Ss^n,$$ 
where 
\begin{equation} \label{cyl_necksize} 
\epsilon_n = \left ( \frac{n(n-4)}{n^2-4} \right )^{\frac{n-4}{8}} \in (0,1).
\end{equation} It follows that 
$$\mathcal{M}_2 = (0,\epsilon_n] \times ((\Ss^n \times \Ss^n 
\backslash \textrm{diag}) /SO(n+1,1)),$$
where the group $SO(n+1,1)$ of conformal transformations acts on each 
$\Ss^n$ factor simultaneously. These metrics corresponding to 
a doubly punctured sphere are all rotationally invariant, and are called the 
Delaunay metrics. We describe them in detail in Section \ref{del_sec}. 

In the present work we explore some of the structure of $\mathcal{M}_k$ 
when $k \geq 3$. Let $\Lambda = \{ p_1, \dots, p_k\}$ with $k\geq 3$ 
and let $g = u^{\frac{4}{n-4}} \overset{\circ}{g} \in \mathcal{M}_\Lambda$. 
As it happens, the metric $g$ is asymptotic to a Delaunay metric near 
each puncture $p_j$, and so one can associate a Delaunay parameter 
$\epsilon_j(g) \in (0, \epsilon_n]$ to each $p_j$ and $g \in \mathcal{M}_\Lambda$. 
(See Section \ref{asymp_sec}.) 
Our main compactness theorem is the following. 
\begin{thm} \label{cmptness_thm} 
Let $k\geq 3$ and let $\delta_1>0, \delta_2>0$ be positive numbers. 
Then the set 
$$\Omega_{\delta_1, \delta_2} = \{ g \in \mathcal{M}_k : 
\operatorname{dist}_{\overset{\circ}{g}} (p_j, p_l) \geq \delta_1 
\textrm{ for each }j \neq l, \epsilon_j(g) \geq \delta_2 \}$$ 
is sequentially compact in the Gromov--Hausdorff topology. 
\end{thm} 
We model this result on a compactness theorem 
of Pollack \cite{Pol},  which states that the similarly 
defined set in the moduli space of singular Yamabe metrics 
is sequentially compact. Very recently Wei \cite{Wei} 
proved a similar theorem in the context of constant 
$\sigma_k$-curvature. Pollack's theorem was an important 
first step in understanding the structure of the moduli space of 
singular Yamabe metrics on a finitely puctured sphere,  
a program that is still not complete. We hope our theorem 
above can play a similar role in advancing the general 
theory of constant $Q$-curvature metrics with isolated 
singularities. 

\section{Preliminaries} 

In this section we present some prerequisite analysis proven 
elsewhere which we will use below. 

We first rewrite \eqref{const_q_curv_eq}. Pulling back by (the inverse of) 
stereographic projection, we can write 
\begin{equation} \label{sph_soln}
\overset{\circ}{g} = \left ( \frac{2}{1+|x|^2} \right )^2 \delta = 
U_{\textrm{sph}}^{\frac{4}{n-4}} \delta, \qquad 
U_{\textrm{sph}} = \left ( \frac{1+|x|^2}{2} \right )^{\frac{4-n}{2}}. 
\end{equation} 
In these coordinates \eqref{const_q_curv_eq} takes the form 
\begin{equation} \label{const_q_curv_eq2} 
u : \R^n \backslash \{ q_1, \dots, q_k \} \rightarrow (0,\infty), \qquad 
\Delta_0^2(U_{\textrm{sph}} u) = \frac{n(n-4)(n^2-4)}{16} 
(U_{\textrm{sph}}u)^{\frac{n+4}{n-4}}, 
\end{equation} 
where $\Delta_0$ is the usual flat Laplacian and $q_j$ is the 
image of $p_j$ under the stereographic map. 
Also, the 
condition $R_g \geq 0$ is equivalent to the differential inequality 
\begin{equation} \label{pos_scal}
-\Delta_0 (U_{\textrm{sph}} u)^{\frac{n-2}{n-4}} \geq 0 \Leftrightarrow 
-\Delta_0 (U_{\textrm{sph}}u) \geq \frac{2}{n-4} \frac{|\nabla 
(U_{\textrm{sph}}u)|^2}{U_{\textrm{sph}}u} . 
\end{equation} 

In this Euclidean setting the transformation rule \eqref{trans_rule1} 
reads 
\begin{equation} \label{scaling_law} 
\Delta_0^2 u = A u^{\frac{n+4}{n-4}} \Rightarrow u_\lambda (x) 
= \lambda^{\frac{n-4}{2}} u(\lambda x) \textrm{ satsfies } \Delta_0^2 
u_\lambda = A u_\lambda^{\frac{n+4}{n-4}} \textrm{ for each }\lambda >0
\end{equation} 
for any constant $A$.

\subsection{Delaunay metrics} \label{del_sec} 

Let $p \neq q \in \Ss^n$ and let $g \in \mathcal{M}_{\{ p,q\}}$. We may 
precompose by an appropriate dilation and assume $p=-q$, and 
then rotate $\Ss^n$ so that $p$ is the north pole and $q$ is 
the south pole. After reframing as in the previous paragraph we 
obtain a function $u : \R^n \backslash \{ 0 \} \rightarrow (0, \infty)$ 
satisfying the PDE \eqref{const_q_curv_eq2}. 
Lin \cite{Lin} proved that this solution must be rotationally invariant 
about $0$, and later Frank and K\"onig \cite{FK}
classified all the ODE solutions. 

Their classification is easiest to see after changing to cylindrical 
coordinates. We let 
\begin{eqnarray} \label{cyl_variables} 
t = -\log |x|,&& \quad \theta = \frac{x}{|x|}, \\ \nonumber 
v(t,\theta) & = &  
e^{\left ( \frac{4-n}{2} \right ) t} U_{\textrm{sph}}(e^{-t} \theta) 
u(e^{-t} \theta) = (e^t \cosh t)^{\frac{4-n}{2}} u(e^{-t} \theta) 
\end{eqnarray} 
This transforms the Paneitz operator into 
\begin{eqnarray} \label{cyl_paneitz_op}
P_{\textrm{cyl}} &= &\frac{\partial^4}{\partial t^4} + \Delta_\theta^2 + 
2\Delta_\theta \frac{\partial^2}{\partial t^2} - \left ( \frac{n(n-4)+8}{2}
\right ) \frac{\partial^2}{\partial t^2}  \\ \nonumber 
&&- \frac{n(n-4)}{2} \Delta_\theta + \frac{n^2(n-4)^2}{16}, 
\end{eqnarray}
so that \eqref{const_q_curv_eq2} becomes 
\begin{equation} \label{cyl_const_q_curv_eq}
v: \R \times \Ss^{n-1} \rightarrow (0,\infty), \qquad 
P_{\textrm{cyl}} v = \frac{n(n-4)(n^2-4)}{16} v^{\frac{n+4}{n-4}}.
\end{equation} 
The fact that the orginal function $u$ is radial implies 
$v$ is a function of $t$ alone, and so \eqref{cyl_const_q_curv_eq}
reduces to the ODE  
\begin{equation} \label{del_ode}
\ddddot v - \left ( \frac{n(n-4)+8}{2} \right ) \ddot v + \frac{n^2(n-4)^2}{16} v 
= \frac{n(n-4)(n^2-4)}{16} v^{\frac{n+4}{n-4}}.
\end{equation} 

We find two solutions explicitly. The cylindrical solution is the 
only constant solution, namely $v_{\textrm{cyl}}=\epsilon_n$, given 
in \eqref{cyl_necksize}. Also, the spherical solution $U_{\textrm{sph}}$ 
given in \eqref{sph_soln} 
transforms under the change of variables \eqref{cyl_variables} 
into $v_{\textrm{sph}}=(\cosh t)^{\frac{4-n}{2}}$. The Delaunay solutions found 
by Frank and K\"onig in \cite{FK} interpolate between the 
cylindrical and spherical solutions. Indeed, for each $\epsilon 
\in (0,\epsilon_n)$ there exists a unique solution $v_\epsilon$ of 
the ODE \eqref{del_ode} realizing its minimal value of $\epsilon$ 
at $t=0$. Each $v_\epsilon$ is periodic with minimal period 
$T_\epsilon$, and these Delaunay solutions account for all 
global solutions of the ODE \eqref{del_ode}. 

Transforming back to Euclidean coordinates, we of course obtain 
the solutions 
\begin{equation} \label{del_eucl_coords}
u_\epsilon : \R^n \backslash \{ 0 \} \rightarrow (0,\infty), 
\qquad u_\epsilon (x) = |x|^{\frac{4-n}{2}} v_\epsilon (-\log |x|).
\end{equation} 
We may then apply global conformal transformations to construct
the translated Delaunay solutions. The first such 
family is $\widetilde u_{\epsilon,a}(x) = u_\epsilon (x-a)$ 
for some fixed vector $a \in \R^n$. The second family is 
more important to our later analysis, and is given by 
translating the point at infinity. More precisely, we define 
\begin {eqnarray*} 
u_{\epsilon, a} (x) & = & \widehat{\mathbb{K}}_0 (\widehat{\mathbb{K}}_0
(u_\epsilon (\cdot - a)) (x)  \\
& = & |x|^{\frac{n-4}{2}} \left | \frac{x}{|x|} - |x| a\right |^{\frac{4-n}{2}}
v_\epsilon \left ( -\log |x| - \log \left | \frac{x}{|x|} - |x| a \right | \right )  ,
\end {eqnarray*}
where 
$$\widehat {\mathbb{K}}_0(u)(x) = |x|^{4-n} u \left ( \frac{x}{|x|^2} \right )$$ 
is the Kelvin transform of $u$. In cylindrical coordinates we can 
write this expression for $u_{\epsilon,a}$ as 
\begin{eqnarray} \label{trans_del}
v_{\epsilon,a} (t,\theta) & = & |\theta - e^{-t} a|^{\frac{4-n}{2}}
v_\epsilon(t + \log |\theta - e^{-t} a|) \\ \nonumber 
& = & v_\epsilon(t) + e^{-t} \langle \theta, a\rangle 
\left ( -\dot v_\epsilon(t) + \frac{n-4}{2} v_\epsilon(t) \right ) 
+ \mathcal{O}(e^{-2t})
\end{eqnarray} 

\subsection{Asymptotics} \label{asymp_sec}

In \cite{JX} Jin and Xiong proved that any positive, 
superharmonic solution of \eqref{const_q_curv_eq2} 
in a punctured ball is asymptotically symmetric. In 
other words, they show there exists $\alpha >0$ such that 
\begin{equation} \label{asymp_symm1} 
u(x) = \overline{u}(|x|) (1+ \mathcal{O}(|x|^\alpha)), \qquad 
\overline{u} (r) = \frac{1}{r^{n-1} |\Ss^{n-1}|} \int_{|x|=r} 
u(\theta) d\theta.
\end{equation} 
Later the first two authors \cite{AO} and the third 
author \cite{R} independently derived refined asymptotics for 
positive, singular solutions of \eqref{const_q_curv_eq2}. 
Roughly speaking, the translated Delaunay solutions 
of \eqref{trans_del} give the next order term in the expansion of 
$u$. They show there exists $\beta > 1$, $\epsilon \in (0, \epsilon_n]$, 
$T \in [0,T_\epsilon)$, 
and $a \in \R^n$ such that 
\begin{eqnarray} \label{refined_asymp1}
u(x) & = & |x|^{\frac{4-n}{2}} \left ( v_\epsilon (-\log |x| +T)\right. \\ \nonumber 
& + & \left.|x| \left \langle \frac{x}{|x|}, a
\right \rangle \left ( -\dot v_\epsilon (-\log |x| +T) + \frac{n-4}{2} v_\epsilon(-\log|x| +T)
\right ) + \mathcal{O}(|x|^\beta) \right ). 
\end{eqnarray} 
In cylindrical coordinates this estimate has the form 
\begin{equation} \label{refined_asymp2} 
v(t,\theta) = v_\epsilon(t +T) + e^{-t} \langle \theta, a \rangle \left ( -\dot v_\epsilon
(t+T) + \frac{n-4}{2} v_\epsilon(t+T) \right ) + \mathcal{O}(e^{-\beta t}).
\end{equation}  

\subsection{Some other useful theorems} 
For the sake of  completeness, we state some background results which will be required later in the proof of our main result.

We first quote the following theorem of Chang, Han and Yang \cite[Theorem~1.1]{CHY}.
\begin{thm}[Chang, Han and Yang] \label{convex_bndry_lemma}
Let $n\geq 5$, let $\Lambda \subset \Ss^n$ be a proper closed set, 
and let $g = u^{\frac{4}{n-4}} \overset{\circ}{g}$ be a complete metric
on $\Ss^n \backslash \Lambda$ such that 
$$Q_g = \frac{n(n^2-4)}{8}, \qquad R_g \geq 0.$$ 
Then $\partial \overset{\circ}{\B}_\rho(x_0)$ is has positive 
mean curvature with respect to $g$, computed with the 
inward pointing normal, where $\overset{\circ}{\B}_\rho(x_0)$
is any ball with respect to the round metric contained in $\Ss^n \backslash 
\Lambda$.  
\end{thm} 

We will also need a version of Harnack's inequality, which was proven 
by Caristi and Mitidieri \cite[Theorem~3.6]{CM}. 
\begin{thm}[Caristi and Mitidieri]\label{harnack_inequality}  Let $u$ be a superharmonic function defined in 
a domain $\Omega \subset \R^n$ such that $\Delta_0^2 u = f(u)$, 
where $f$ is either linear or superlinear and $f(0) = 0$. Then there exists 
$\rho_0> 0$ such that for $\rho  \leq \rho_0$ we have 
\begin{equation} \label{harnack} 
\sup_{\overset{\circ}{\B}_\rho(p)} u \leq C \inf_{\overset{\circ}{\B}_\rho(p)} u,
\end{equation} 
where the constant $c$ depends only on the domain $\Omega$, the function $f$, 
and $\rho$. 
\end{thm} 

Gursky and Malchiodi \cite[Proposition~2.5]{GM} prove the existence of a positive Greens function 
for the Paneitz operator of the round sphere. 
\begin{thm}[Gursky and Malchiodi]  \label{greens_func}
Let $(M,g)$ be a compact 
Riemannian manifold such that $R_g \geq 0$ and $Q_g >0$. Then 
for each $p \in M$ there exists a Greens function $G_p$ 
satisfying 
$$G_p : M \backslash \{ p \} \rightarrow (0,\infty), \qquad 
P_gG_p = \delta_p,$$ 
where $\delta_p$ is the Dirac $\delta$-function with a 
singularity at $p$. Furthermore, if either $n=5,6,7$ or $g$ is conformally flat 
then there exists $c>0$ depending only on $n$ and $\alpha$ such that 
\begin{equation} \label{green_estimate} 
G_p(x) = \frac{1}{2n(n-2)(n-4) \omega_n} \left ( \operatorname{dist}_g(x,p) \right )^{4-n} +
\mathcal{O}(1) \end{equation} 
in conformal normal coordinates, 
where $\omega_n$ is the volume of a unit ball in $\R^n$. 
\end{thm}

\section{Pohozaev invariants} 

One often finds integral invariants in geometric variational 
problems. The reader can find a general abstract framework for constructing 
these invariants in the paper by Gover and \O rsted \cite{GO}. In 
future work we will explicitly write out the full Pohozaev 
invariant using the first variation tensor defined in \cite{LY}.


We consider a function 
$$v:(a,b) \times \Ss^{n-1} \rightarrow \R$$ 
satisfying \eqref{cyl_const_q_curv_eq}. Given such a 
function $v$ we define the Hamiltonian functional 
\begin{eqnarray} \label{poho_integrand1} 
\mathcal{H} (v) & = & - \frac{\partial v}{\partial t} \frac{\partial^3 v}{\partial t^3}
+ \frac{1}{2} \left ( \frac{\partial^2 v}{\partial t^2} \right )^2 - \frac{1}{2} (\Delta_\theta v)^2
- \left |\nabla_\theta \frac{\partial v}{\partial t} \right |^2 + \frac{n(n-4)}{4} |\nabla_\theta v|^2 
\\ \nonumber 
&& + \left ( \frac{n(n-4)+8}{4} \right ) \left ( 
\frac{\partial v}{\partial t} \right )^2- \frac{n^2(n-4)^2}{32} v^2 + \frac{(n-4)^2(n^2-4)}{32} 
v^{\frac{2n}{n-4}}.
\end{eqnarray}
Integrating by parts we find 
\begin{equation} \label{poho_integral0}
\frac{d}{dt} \int_{\{ t \} \times \Ss^{n-1}} \mathcal{H} (v) d\theta = 0,
\end{equation} 
which allows us to define our first integral invariant as 
\begin{equation} \label{poho_integral1} 
\widetilde {\mathcal{P}}_{\textrm{rad}} (v) = \int_{\{ t \} \times \Ss^{n-1}} 
\mathcal{H} (v) d\theta. 
\end{equation} 

Now we can define the radial (or dilational) Pohozaev invariants associated to 
a metric $g \in \mathcal{M}_k$ at each puncture point $p_j$. Recall 
that  $g = u^{\frac{4}{n-4}}
\overset{\circ}{g}$ is a complete, conformally flat metric on $\Ss^n \backslash 
\{ p_1, \dots, p_k\}$ with $Q_g = \frac{n(n^2-4)}{8}$ and $R_g \geq 0$. 
Completeness forces 
$$\liminf_{\dist_{\overset{\circ}{g}} (x, p_j) \rightarrow 0} u(x) = \infty$$
 for each $j$, while $Q_g = \frac{n(n^2-4)}{8}$ is equivalent to the PDE 
 \eqref{const_q_curv_eq2}, after stereographic projection 
 down to $\R^n \backslash \{ p_1, \dots, p_k\}$. Choose coordinates centered 
 at one of the punctures $p_j$, and the perform the cylindrical 
 change of variables \eqref{cyl_variables}, which gives us a function 
 $$v: (A,\infty) \times \Ss^{n-1} \rightarrow (0,\infty)$$ 
 satisfying \eqref{cyl_const_q_curv_eq}. We define 
 $$\mathcal{P}_{\textrm{rad}} (g, p_j) = 
 \widetilde{\mathcal{P}}_{\textrm{rad}} (v) = \int_{\{ t \} \times 
 \Ss^{n-1}} \mathcal{H}(v) d\theta,$$ 
 which is well-defined by \eqref{poho_integral0}. 
 
 In the special case that $k=2$ we can evaluate the dilational 
 Pohozaev invariant more explicitly. In this situation we may as well 
 let the puncture points be the north and south poles, and thus we obtain a function 
 $$v = v_\epsilon : \R \times \Ss^{n-1} \rightarrow (0,\infty)$$ 
 satisfying \eqref{del_ode}. Thus the Hamiltonian \eqref{poho_integrand1} 
 reduces to 
 \begin{equation} \label{poho_integrand2} 
 \overline{\mathcal{H}} (v) = -\dot v \dddot v + \frac{1}{2} \ddot v^2 + 
 \left ( \frac{n(n-4)+8}{4} \right ) \dot v^2 - \frac{n^2(n-4)^2}{32} v^2 
 - \frac{(n-4)^2(n^2-4)}{32} v^{\frac{2n}{n-4}}.
 \end{equation} 
 Moreover, because $\overline{\mathcal{H}}$ does not depend on 
 $\theta$ and its integral over a sphere does not depend on $t$, 
 this reduced Hamiltonian must be constant on solutions 
 of \eqref{del_ode}. (One can, of course, explicitly verify this 
 constancy by taking a derivative.)
 
 A computation reveals 
 $$\overline{\mathcal{H}} (v_{\textrm{sph}}) = 0, \qquad 
 \overline{\mathcal{H}} (v_{\epsilon_n}) = -\frac{(n-4)(n^2-4)}{8} 
 \left ( \frac{n(n-4)}{n^2-4} \right )^{n/4} < 0.$$
 Furthermore, Proposition 6 of \cite{vdB} implies the Delaunay 
 solutions are ordered (in fact, uniquely determined!) by their 
 energy \eqref{poho_integrand2}. 
 
 Combining our analysis above with  \eqref{refined_asymp1} 
 and \eqref{refined_asymp2} we immediately see the following 
 \begin{lemma} 
 Let $g = u^{\frac{4}{n-4}}
\overset{\circ}{g}$ be a complete, conformally flat metric on $\Ss^n \backslash 
\{ p_1, \dots, p_k\}$ with $Q_g = \frac{n(n^2-4)}{8}$ and $R_g \geq 0$. For 
each puncture $p_j$ define $\mathcal{P}_{\textrm{rad}}(g,p_j)$ as 
above. Then $\mathcal{P}_{\textrm{rad}}(g,p_j)< 0$ 
and depends only on the necksize $\epsilon_j$ of the Delaunay 
asymptote at $p_j$. Moreover, decreasing $\epsilon_j$ will 
increase $\mathcal{P}_{\textrm{rad}}(g,p_j)$ towards $0$. In 
particular, bounding the radial Pohozaev invariants 
$\mathcal{P}_{\textrm{rad}} (g,p_j)$ away from zero is equivalent to 
bounding the necksizes $\epsilon_j$ away from zero. 
\end{lemma} 

\begin{proof} We have shown that 
$$\mathcal{P}_{\textrm{rad}}(g,p_j)
= \int_{\{ t \} \times \Ss^{n-1}} 
\mathcal{H}(v) d\theta$$
is well-defined, because the integral does not depend 
on our choice of $t$. Now let $t \rightarrow \infty$, 
and observe that $v \rightarrow v_{\epsilon_j}$, the 
Delaunay asymptote of $g$ at the puncture $p_j$. In 
particular, $\mathcal{H}(v) \rightarrow \overline{\mathcal{H}}
(v_{\epsilon_j})$, We conclude that 
$$\mathcal{P}_{\textrm{rad}}
(v) = \mathcal{P}_{\textrm{rad}}(v_{\epsilon_j}) = 
|\Ss^{n-1}| \overline{\mathcal{H}}(v_{\epsilon_j}) < 0.$$
The remainder of the lemma follows from the energy 
ordering theorem of van den Berg \cite{vdB} applied to 
the Delaunay solutions, as described in the paragraph above. 
\end{proof}

\begin{rmk} Our  radial Pohozaev invariant is basically the 
same as the one defined in Proposition 4.1 of \cite{ADM}.
Jin and Xiong \cite{JX} write out the same invariant for higher order 
equations. 
\end{rmk} 

It will actually be useful for later computations to decompose the 
Hamiltonian energy $\mathcal{H}$ given in \eqref{poho_integrand1} 
as 
\begin{equation} \label{poho_integrand3} 
\mathcal{H} (v) = \mathcal{H}_{\textrm{cyl}}(v) + 
\frac{(n-4)^2(n^2-4)}{32} v^{\frac{2n}{n-4}}. 
\end{equation} 
The same computation as in \eqref{poho_integral1} 
shows the following lemma. 
\begin{lemma} \label{poho_general_const} 
Let $v$ satisfy 
$$v: (a,b) \times \Ss^{n-1} \rightarrow \R, \qquad 
P_{\textrm{cyl}}v = A v^{\frac{n+4}{n-4}}$$ 
for some constant $A$. Then the integral 
$$\int_{\{ t \} \times \Ss^{n-1}} 
\mathcal{H}_{\textrm{cyl}} (v) + \frac{(n-4)}{2n} A
v^{\frac{2n}{n-4}} d\theta$$ 
does not depend on $t$. 
\end{lemma}

\section{Proof of the compactness theorem}

In this section we prove Theorem \ref{cmptness_thm}. 
We first use standard blow-up techniques to prove {\it a priori} bounds on the 
$\mathcal{C}^4$-norm of solutions of \eqref{gen_sing_q_curv}. Once we obtain 
these bounds, we use them to extract a convergent subsequence. Finally we 
prove that our limit is non-trivial, using the fact that the radial Pohozaev 
invariants of our original sequence 
of metrics remain bounded away from zero. 

\subsection{{\it A priori} bounds} 

We prove some {\it a priori} bounds for solutions of 
$$\overset{\circ}{P} u = \frac{n(n-4)(n^2-4)}{16} 
u^{\frac{n+4}{n-4}}.$$ 

\begin {thm} \label{upper_bnd} 
Let $n \geq 5$, let $\Lambda \subset \Ss^n$ be a proper
closed set, and let $g = u^{\frac{4}{n-4}} \overset{\circ}{g}$
be a complete metric on $\Ss^n \backslash \Lambda$ such that 
$$Q_g = \frac{n(n^2-4)}{8}, \qquad R_g \geq 0.$$ 
Then there exists $C>0$ depending only on the dimension $n$ such that 
\begin{equation} \label{apriori_upper_bnd} 
u(x) \leq C \left ( \operatorname{dist}_{\overset{\circ}{g}} (x, \Lambda)\right )^{\frac{4-n}{2}} .
\end{equation} 
\end{thm} 

\begin{rmk} In the context of $g \in \mathcal{M}_k$ with $k \geq 2$, 
our upper bound \eqref{apriori_upper_bnd} is 
very similar to, but slightly stronger than, Proposition 3.2 of \cite{JX}, 
because our constant $C$ depends only on the dimension $n$. 
\end{rmk} 

\begin {proof} Our proof borrows from Pollack's proof of 
the corresponding upper bound in the scalar curvature case. 

Given any $g$ satisfying the hypotheses of Theorem \ref{upper_bnd}, 
$x_0 \not \in \Lambda$, and $\rho >0$ such that $\overset{\circ}{\B}_\rho
\subset \Ss^n \backslash \Lambda$ we define the auxiliary 
function 
\begin{equation} \label{help_fnc1}
f: \overset{\circ}{\B}_\rho \rightarrow \R, \qquad 
f(x) = (\rho- \operatorname{dist}_{\overset{\circ}{g}} 
(x, x_0))^{\frac{n-4}{2}} u(x) .
\end{equation} 
Observe that choosing $\rho = \frac{1}{2} 
\operatorname{dist}_{\overset{\circ}{g}} (x_0, \Lambda)$ 
yields 
\begin{equation} \label{help_fnc2} 
f(x_0) = \rho^{\frac{n-4}{2}} u(x_0) = \left ( \frac{1}{2} 
\operatorname{dist}_{\overset{\circ}{g}} (x_0, \Lambda) \right
)^{\frac{n-4}{2}} u(x_0) ,
\end{equation} 
so it will suffice to find $C$ depending only on $n$ such that 
$f(x) \leq C$ for all admissible choices of $\Lambda$, $u$, 
$x_0$, and $\rho$. 

We suppose the contrary and derive a contradiction. To this end, 
let $\Lambda_i$, $u_i$, $x_{0,i}$ and $\rho_i$ be admissible 
as described above and suppose 
\begin{equation} \label{blow_up1}
M_i = f(x_{1,i}) = \sup_{x \in \overset{\circ}{\B}_{\rho_i}(x_{0,i})} f(x) 
\rightarrow \infty . 
\end {equation} 
Observe that $\left. f \right |_{\partial \overset{\circ}{\B}_{\rho_i} 
(x_{0,i})} = 0$, so $x_{1,i}$ must lie in the interior of 
the ball $\overset{\circ}{\B}_{\rho_i} (x_{0,i})$. 
Next let 
$$r_i = \rho_i - \operatorname{dist}_{\overset{\circ}{g}}
(x_{1,i}, x_{0,i}),$$ 
let $y$ be geodesic normal coordinates centered at $x_{1,i}$, 
and define 
\begin{equation} \label{blow_up2} 
\lambda_i = 2 (u_i(x_{1,i}))^{\frac{4-n}{2}}, \qquad 
R_i = \frac{r_i}{\lambda_i} = \frac{r_i}{2} (u_i (x_{1,i}))^{\frac{n-4}{2}} 
= \frac{1}{2} M_i^{\frac{2}{n-4}} 
\end{equation} 
and 
\begin{equation} \label{blow_up3} 
w_i : \B_{R_i} (0) \rightarrow (0,\infty), \qquad w_i (y) = 
\lambda_i^{\frac{n-4}{2}} u_i(\lambda y).
\end{equation} 
By \eqref{trans_rule1} (or, equivalently, \eqref{scaling_law}) the function $w_i$ 
solves 
$$P_{\lambda \overset{\circ}{g}}w_i = \frac{n(n-4)(n^2-4)}{16} 
w_i ^{\frac{n+4}{n-4}}. $$
Moreover, by construction 
$$2^{\frac{n-4}{2}} = w_i(0) = \sup_{\B_{R_i}(0)} w_i(x) .$$
Using the Arzela-Ascoli theorem we extract a subsequence, which 
we still denote by $w_i$, that converges uniformly on 
compact subsets of $\R^n$. Furthermore, as $i \rightarrow 
\infty$ the rescaled metrics $\lambda \overset{\circ}{g}$ 
converge to the Euclidean metric. Therefore, in the limit we 
obtain a function 
\begin{equation} \label{blow_up4} 
\overline w : \R^n \rightarrow [0,\infty) , \quad \Delta_0^2 \overline{w} 
= \frac{n(n-4)(n^2-4)}{16} \overline{w}^{\frac{n+4}{n-4}}, \quad 
\overline{w}(0) = \sup \overline{w} = 2^{\frac{n-4}{2}}.
\end{equation}
By the classification theorem in \cite[Theorem~1.3]{Lin} we must have 
$$\overline w(x) = \left ( \frac{1+|x|^2}{2} \right )^{\frac{4-n}{2}}.$$

Thus each solution $u_i$ has a \quotes{bubble} when $i$ is 
sufficiently large, that is for $i$ 
sufficiently large a small neighborhood of $x_{1,i}$ is close
(in $\mathcal{C}^4$-norm) to the round metric, and hence 
has a concave boundary. We verify this by computing the 
mean curvature of a geodesic sphere explicitly. The round metric 
has the form $g_{lm} = \frac{4}{(1+|x|^2)^2} \delta_{lm}$ in stereographic 
coordinates, and 
in general the mean curvature of a hypersurface $\Sigma$ in a Riemannian 
manifold with unit normal $\eta$  is given by 
\begin {equation} \label{gen_mean_curv} 
H_\Sigma = - \operatorname{tr}_g \langle \nabla_{\partial l} \eta, \partial_m \rangle 
= -\partial_l \eta^l - \eta^p \Gamma_{lp}^l.\end{equation} 
A geodesic sphere centered at $0$ coincides with a Euclidean round sphere 
centered at the origin (with a different radius, of course), and so the inward 
normal vector is 
$$\eta = - \left ( \frac{1+|x|^2}{2|x|} \right ) x^l \partial_l .$$
A computation shows 
$$H = -2n|x|(1+|x|^2)  + \frac{ n-1 + n|x|^2}{|x|},$$
which is negative, in particular, when $|x|>3$. Additionally, since 
$\| w_i - \overline{w}\|_{\mathcal{C}^4(\B_{\frac{3R_i}{4}} (0))}$ is arbitrarily 
small when $i$ is sufficiently large, we see that $\partial \B_{\frac{3R_i}{4}}
(0)$ is also mean concave with respect to the metric $w_i^{\frac{4}{n-4}}
\delta_{lm}$, which in turn implies $\partial \B_{\frac{3|x_{1,i}|}{8}} (x_{1,i})$ is mean concave 
with respect to the metric induced by $u_i^{\frac{4}{n-4}} \delta_{lm}$. This 
contradicts Theorem \ref{convex_bndry_lemma}.
\end{proof}

We immediately obtain the following Corollary. 
\begin{cor} \label{upper_bnd2}
For each compact subset $\Omega \subset 
\Ss^n \backslash \Lambda$, $l \in \mathbb{N}$ and $\alpha \in (0,1)$ 
there exists $C_1$ 
depending only on $n$, $l$, $\Omega$, and $\alpha$ such that 
\begin{equation} \label{apriori_upper_bnd2}
\| u_i  \|_{\mathcal{C}^{l,\alpha} (\Omega)} \leq C_1.
\end{equation} 
\end{cor}

We also record here a lower bound due to Jin and Xiong \cite[Theorem~1.3]{JX}. 
\begin{thm} [Jin and Xiong] \label{rmbl_sing} 
Let 
$$v:[A,\infty) \times \Ss^{n-1} \rightarrow (0,\infty)$$ 
solve \eqref{cyl_const_q_curv_eq}. Then $\mathcal{P}_{\textrm{rad}}
(v) \leq 0$ with equality if and only if 
$$\liminf_{t \rightarrow \infty} v(t,\theta) = \limsup_{t \rightarrow \infty} 
v(t,\theta) = \lim_{t \rightarrow \infty} v(t,\theta) = 0.$$
Otherwise, if $\mathcal{P}_{\textrm{rad}} (v) < 0$, there exists $C_2> 0$
(which depends on the solution $v$!) such that $v(t,\theta) \geq C_2$. 
\end{thm} 

\begin{cor} Let $g = u^{\frac{4}{n-4}} \overset{\circ}{g} \in \mathcal{M}_k$
have the singular set $\Lambda = \{ p_1, \dots, p_k\}$. 
Then there exists $C_2>0$ (depending on the solution $u$!) such that 
$$u(x) \geq C_2 \left(\min_{1 \leq j \leq k} 
\operatorname{dist}_{\overset{\circ}{g}} (x,p_j) \right)^{\frac{4-n}{2}}.$$
\end{cor}

\subsection{Sequential compactness}

In this section we complete our proof of 
sequential compactness. 
To this end, let 
$$\{ g_i  = u_i^{\frac{4}{n-4}} \overset{\circ}{g} \} \subset 
\Omega_{\delta_1, \delta_2} \subset\mathcal{M}_k$$
and denote the singular set  of the conformal factor 
$u_i$ by $\Lambda_i = \{ p_1^i , \dots, p_k^i\}$.

The following lemma will simplify our later analysis since it allows us to 
assume the singular points are fixed. 

\begin{lemma} \label{conv_punctures}
Let $g_i = u_i^{\frac{4}{n-4}} \overset{\circ}{g}$
be a sequence in $\mathcal{M}_k$ as described above. After passing to a 
subsequence, we may assume that when 
$i$ is sufficiently large both $g_i$ and $u_i$ are regular on the 
$$\Ss^n \backslash \left ( \cup_{j=1}^k \overset{\circ}{\B}_{\delta_1/2}
(p_j^i) \right ),$$
where $\overset{\circ}{\B}_r(p)$ is the geodesic ball centered at 
$p$ with radius $r$, with respect to the round metric $\overset{\circ}{g}$. 
\end{lemma} 

\begin{proof} The set 
$$(\Ss^n)^k \backslash \left \{(q_1, \dots, q_k) \in (\Ss^n)^k : 
\dist_{\overset{\circ}{g}} (q_j, q_l) \geq \delta_1 \textrm{ for each }
j \neq l \right \}$$ 
is compact and contains each singular set $\Lambda_i = \{ p_1^i, \dots, 
p_k^i\}$. Thus we may extract a convergent subsequence, 
which we still denote as $\Lambda_i = \{ p_1^i, \dots, p_k^i\}$, 
with $p_j^i \rightarrow \bar p_j$. The 
lemma now follows from $p_j^i \rightarrow \bar p_j$ for each $j$. 
\end{proof} 

To set notation, we define the compact sets 
\begin{equation} \label{unif_cmpt_set} 
K_m = \Ss^n \backslash \left ( \cup_{j=1}^k 
\overset{\circ}{\B}_{2^{-m} \delta_1} (\bar p_j) \right ) 
\end{equation}  
for each natural number $m \in \mathbf{N}$. By construction the family 
$\{ K_m\}$ is a compact exhaustion of $\Ss^n \backslash \{ \bar p_1, 
\dots, \bar p_k\}$. Furthermore, by the convergence $p_j^i \rightarrow 
\bar p_j$, for each fixed $m$ there exists 
$i_0$ such that $i \geq i_0$ implies $u_i$ is smooth in $K_m$. Therefore, 
combining Corollary~\ref{upper_bnd2} and the Arzela-Ascoli theorem 
we obtain a convergent subsequence, which we again denote by $u_i$, 
that converges uniformly on compact subsets of $\Ss^n \backslash 
\overline{\Lambda}$ to a limit $\overline{u}$, where $\overline{\Lambda} =
\{ \bar p_1, \dots, \bar p_k\}$. Furthermore, combining our 
{\it a priori} bounds and elliptic regularity, we see that the 
limit function satisfies 
\begin{equation} \label{limit_eqn} 
\overline{u}: \Ss^n \backslash \overline{\Lambda} \rightarrow [0,\infty), \quad 
\overset{\circ}{P}\overline{u} = \frac{n(n-4)(n^2-4)}{16} \overline{u}^{\frac{n+4}{n-4}}.
\end{equation} 

\begin{prop} The limit function $\overline{u}$ constructed in the 
paragraph above is positive on $\Ss^n \backslash \{ \bar p_1, 
\dots, \bar p_k\}$. 
\end{prop} 

\begin{proof} If the proposition does not hold then there exists 
$q \in \Ss^n \backslash \{ \bar p_1, \dots, \bar p_k\}$ such that 
$$0 = \overline{u} (q) = \lim_{i \rightarrow \infty} u_i(q).$$
Let $\epsilon_i = u_i(q)$ and 
$$w_i : \Ss^n \backslash \{ p_1^i, \dots, p_k^i\} \rightarrow (0, \infty), \qquad 
w_i (x) = \frac{1}{\epsilon_i} u_i(x).$$ 
As a consequence of \eqref{scaling_law}, we have 
\begin{equation} \label{revised_scaling1} 
\overset{\circ}{P} w_i = \epsilon_i^{\frac{8}{n-4}} \frac{n(n-4)(n^2-4)}{16} 
w_i^{\frac{n+4}{n-4}}.\end{equation} 
In addition, $w_i$ satisfies the normalization 
\begin{equation} \label{normalization} 
w_i(q) = 1\end{equation} 
for each $i$ by construction. 

By \eqref{apriori_upper_bnd2}, for each $m \in \mathbf{N}$ there 
exists $C_1$ depending on $m$ and the dimesnion $n$ such that 
\begin{equation} \label{final_step1}
\sup_{K_m} u_i \leq C_1. \end{equation}
Next we find an upper bound for $1/\epsilon_i$. Using the fact that 
$0<U_{\textrm{sph}} \leq 2^{\frac{n-4}{2}}$ and the Harnack 
inequality in Theorem~\ref{harnack_inequality} we get
$$2^{\frac{n-4}{2}} \epsilon_i \geq \inf_{K_m} (U_{\textrm{sph}} u_i) 
\geq \frac{1}{\widetilde C(m,n)} \sup_{K_m} (U_{\textrm{sph}} u_i) 
= C_2,$$ 
and so 
\begin{equation} \label{final_step2} 
\frac{1}{\epsilon_i} \leq C_3. \end{equation}  
Combining \eqref{final_step1} and \eqref{final_step2} 
we obtain a uniform upper bound 
$$\sup_{K_m} w_i \leq C_3,$$ 
where $C_3$ depends only on $n$ and $m$. 

We conclude that $w_i$ converges uniformly 
on compact subsets of $\Ss^n \backslash \overline{\Lambda}$
to a function 
$$\overline{w} : \Ss^n \backslash \overline{\Lambda} \rightarrow 
[0,\infty) , \qquad 
\overset{\circ}{P}\overline{w} = 0.$$ 
By Theorem \ref{greens_func} we have 
\begin{equation} \label{final_step3} 
\overline{w} = \sum_{j=1}^k \alpha_j G_{\bar p_j},
\end{equation} 
for some coefficients $\alpha_j \geq 0$. By the 
normalization \eqref{normalization} at least one of the 
$\alpha_j$'s is positive, so (in particular) $\overline{w}$ 
is a smooth, positive function. 

Without loss of generality, we may assume $\alpha_1 \neq 0$ 
and center our coordinate system at $\bar p_1$. We now 
use the cylindrical coordinates $t = -\log |x|$ and $\theta = \frac{x}{|x|}$
in a punctured ball centered on $\bar p_1 = 0$, and define 
$$v_i(t,\theta) = e^{\left ( \frac{4-n}{2} \right ) t} u_i(e^{-t} \theta)
U_{\textrm{sph}} (e^{-t} \theta) , \quad 
z_i (t,\theta) = \frac{1}{\epsilon_i} v_i (t,\theta)  = e^{\left ( \frac{4-n}{2} \right )
t} w_i (e^{-t} \theta) U_{\textrm{sph}}(e^{-t}\theta) $$
and 
$$ \overline{v}(t,\theta) = e^{\left ( \frac{4-n}{2} \right )t} \overline{u} (e^{-t} \theta)
(\cosh t)^{\frac{4-n}{2}}, \quad \overline{z}(t,\theta) = e^{\left ( \frac{4-n}{2} 
\right )t} \overline{w} (e^{-t} \theta) (\cosh t)^{\frac{4-n}{2}}. $$
By the expansion \eqref{green_estimate} we have 
\begin{eqnarray} \label{final_step4} 
\overline{z} (t,\theta) & = & e^{\left ( \frac{4-n}{2} \right )t} 
(\cosh t)^{\frac{4-n}{2}} \left ( \frac{\alpha_1}{2n(n-2)(n-4)\omega_n}
e^{\left ( \frac{n-4}{2} \right )t} + \mathcal{O}(1) \right ) \\ \nonumber
& = & \frac{\alpha_1}{2n(n-2)(n-4)\omega_n} + \mathcal{O} 
(e^{(4-n) t}). 
\end{eqnarray} 

Observe that $z_i$ satisfies the PDE 
$$P_{\textrm{cyl}} z_i = \epsilon_i^{\frac{8}{n-4}} \frac{n(n-4)(n^2-4)}{16} 
z_i^{\frac{n+4}{n-4}},$$
so, following Lemma \ref{poho_general_const}, the integral 
$$\int_{\{ t \} \times \Ss^{n-1}} \mathcal{H}_{\textrm{cyl}} (z_i) 
+ \epsilon_i^{\frac{8}{n-4}} \frac{(n-4)^2(n^2-4)}{32} z_i^{\frac{2n}{n-4}}
d\theta$$ 
does not depend on $t$. Moreover, taking a limit as $i \rightarrow \infty$, 
we obtain 
\begin{eqnarray} \label{final_step5} 
& \lim_{i \rightarrow \infty} \int_{\{ t \} \times \Ss^{n-1}} 
\mathcal{H}_{\textrm{cyl}} (z_i) + \epsilon_i^{\frac{8}{n-4} }
\frac{(n-4)^2(n^2-4)}{32} z_i^{\frac{2n}{n-4}} d\theta & 
\\ \nonumber 
= & \int_{ \{ t \} \times \Ss^{n-1}} \mathcal{H}_{\textrm{cyl}} 
(\overline{z}) d\theta & \\ \nonumber 
= & -\int_{ \{ t \} \times \Ss^{n-1}} \frac{n^2(n-4)^2}{32} 
\cdot \frac{\alpha_1^2}{4n^2(n-2)^2(n-4)^2 \omega_n^2} + 
\mathcal{O}(e^{(4-n)t}) & \\ \nonumber 
= &  -\frac{n\alpha_1^2}{128(n-2)^2 \omega_n} + 
\mathcal{O}(e^{(4-n)t}). & 
\end{eqnarray} 

On the other hand, by construction 
\begin{eqnarray*} 
\mathcal{P}_{\textrm{rad}}(v_i) & = & \int_{\{ t \} \times 
\Ss^{n-1}} \mathcal{H}_{\textrm{cyl}}(v_i) + \frac{(n-4)^2(n^2-4)}{32} 
v_i^{\frac{2n}{n-4}} d\theta \\ 
& = & \int_{\{ t \}  \times \Ss^{n-1}} \epsilon_i^2 
\mathcal{H}_{\textrm{cyl}} (z_i)+ \frac{(n-4)^2(n^2-4)}{32}
\epsilon_i^{\frac{2n}{n-4}} z_i^{\frac{2n}{n-4}} d\theta \rightarrow 0,
\end{eqnarray*} 
and so 
$$\lim_{i \rightarrow \infty} \mathcal{P}_{\textrm{rad}} 
(g_i, p_1^i) = 0.$$ 
This contradicts the hypothesis that the asymptotic 
necksizes of $g_i = u_i^{\frac{4}{n-4}} \overset{\circ}{g}$ 
at the puncture points $p_1^i$ are all bounded away from $0$. 
\end{proof}

We finally complete the proof of Theorem \ref{cmptness_thm}. 
\begin{proof} Given a sequence $\{ g_i \} \in \Omega_{\delta_1, \delta_2}$, we 
have already obtained a limit $\overline{g} = 
\overline{u}^{\frac{4}{n-4}} \overset{\circ}{g}$ as a limit of a subsequence. 
We know that 
$$\overline{u} : \Ss^n \backslash \{ \bar p_1, \dots, \bar p_k\} \rightarrow 
(0,\infty), \quad \overset{\circ}{P} \overline{u} = \frac{n(n-4)(n^2-4)}{16} 
\overline{u}^{\frac{n+4}{n-4}},$$ 
where $\bar p_j = \lim_{i \rightarrow \infty} p_j^i$. We have also shown 
that $\overline{u} > 0$ on $\Ss^n \backslash \{ \bar p_1, \dots, 
\bar p_k\}$.

It only remains to verify that $\overline{g}$ is 
complete. If $\overline{g}$ is incomplete then there exists 
$j \in \{ 1, \dots, k\}$ such that 
$$\liminf_{x \rightarrow \bar p_j} \overline{u}(x) < \infty.$$
In this case Theorem \ref{rmbl_sing} implies $\mathcal{P}_{\textrm{rad}} 
(\overline{g}, \bar p_j) = 0$. However, by construction
$$\mathcal{P}_{\textrm{rad}}(\overline{g}, \bar p_j) = \lim_{i \rightarrow \infty} 
\mathcal{P}_{\textrm{rad}} (g_i, p_j^i) \geq \delta_2,$$
giving a contradiction. We conclude that $\overline{g}$ is indeed in  
$\Omega_{\delta_1, \delta_2}$. 
\end{proof} 


\end {document}